\pdfoutput=1
\documentclass[12pt]{article}
\usepackage[utf8]{inputenc}

\usepackage{amsthm,amsfonts,amssymb,amsmath,epsf, verbatim}
\usepackage[hidelinks]{hyperref}

\def\odd{{\rm \, odd}}

\newtheorem{theorem}{Theorem}
\newtheorem{lemma}[theorem]{Lemma}
\newtheorem{corollary}[theorem]{Corollary}
\newtheorem{proposition}[theorem]{Proposition}

\newtheorem{question}[theorem]{Question}
\newtheorem{definition}{Definition}
\newcommand{\seqnum}[1]{\href{https://oeis.org/#1}{\rm \underline{#1}}}

\title{The sum of the reciprocals of the prime divisors of an odd perfect or odd primitive non-deficient number}
\author{Joshua Zelinsky\footnote{Hopkins School, New Haven, CT, USA, zelinsky@gmail.com}}
\date{}

\begin{document}

\maketitle

\begin{abstract}  Write $H(n) = \prod_{p \mid n}p/(p-1)$ where the product is over the prime divisors of $n$. Define $T(n)$ as the sum of the reciprocals of the primes dividing $n$. We prove new bounds for $T(n)$ and $H(n)$ in terms of the smallest prime factor of $n$, under the assumption that $n$ is an odd perfect number. Some of the results also apply under the weaker assumption that $n$ is odd and primitive non-deficient.
    
\end{abstract}

\section{Introduction}

Let $\sigma(n)$ be the sum of the positive divisors of $n$. 
A number is said to be perfect if $\sigma(n)=2n$. For example, $\sigma(6)= 1+2+3+6=12=2(6)$. 

Perfect numbers form sequence \seqnum{A000396} in the OEIS. An old unsolved problem is whether there is any odd integer $n$ which is perfect.  

Given a positive integer $n$, we will write $T(n)$ as the sum of the reciprocals of the primes which divide $n$. For example, $T(12) = 1/2 +1/3= 5/6$. A closely related function is 

$$H(n) = \prod_{p \mid n}\frac{p}{p-1}$$ where the product is over the prime divisors of $n$. The main focus of this paper is to prove new bounds for $T(n)$ and $H(n)$ in terms of the smallest prime factor of $n$, under the assumption that $n$ is an odd perfect number. Some of the results also apply to all odd non-deficient numbers.

The function $T(n)$ is closely related to the arithmetic derivative of $n$. The arithmetic derivative is a function $D(n)$ defined by analogy to the usual derivative from calculus. The function $D(n)$  is the unique function on the natural numbers with the properties that $D(p)=1$ for any prime $p$ and for any integers $m$ and $n$,  one has $D(mn)=D(m)n+mD(n)$.  When $n$ is squarefree, $nT(n) = D(n)$ is the arithmetic derivative of $n$. The function $nT(n)$ is always an integer and is sequence \seqnum{A069359} in the OEIS. The function was brought to wider attention article by Barbeau \cite{Barbeau} but predates Barbeau and apparently has multiple independent discoveries both before and after Barbeau. Tossavainen, Haukkanen,  Merikoski and Mattila  wrote a survey \cite{THMM} article discussing $D(n)$ as well as its pre-Barbeau history.  There has been a large amount of recent work on $D(n)$. Notable work includes Pasten \cite{Pasten1} who used it to give an alternate proof that there are infinitely many primes. Pasten \cite{Pasten2} also created a broad set of generalizations of $D(n)$ which turn out to be closely connected to the ABC conjecture. Recently, Fan and Utev \cite{FanUtev} connected $D(n)$ to an analog of the Lie bracket for arithmetic functions.  Murashka, Goncharenko and Goncharenko \cite{MurashkaGoncharenkoGoncharenko} gave a generalization of the arithmetic derivative to general unique factorization domains.  Finally, Merikoski,  Haukkanen, and Tossavainen \cite{MHT} inspired by the analogy of $D(n)$ to the usual derivative studied the equivalent of differential equations involving $D(n)$. 

We will write $h(n) =\sigma(n)/n$. Note that we also have that $$h(n) = \sum_{d \mid n} \frac{1}{d}.$$ In the literature, $h(n)$ often referred as the {\it{abundancy index}} of $n$.  A number is said to be {\it abundant} when $h(n) >2$, and said to be {\it deficient} when $h(n) <2$. For many purposes, it is useful to group perfect and abundant numbers together as non-deficient numbers. It is not hard to show that 

\begin{equation} h(ab) \geq h(a), \label{basic h inequality}    
\end{equation} with equality if and only if $b=1$. Thus, if $n$ is perfect or abundant, then $mn$ is abundant for any $m >1$.  

Inequality \eqref{basic h inequality} motivates the following definition.
\begin{definition} 
A {\it {primitive non-deficient number}} is a number $n$ which is non-deficient and which has no proper divisor $d$ of $n$ where $d$ is non-deficient.   \end{definition}

The primitive non-deficient numbers form sequence \seqnum{A006039} in the OEIS. All perfect numbers are primitive non-deficient, and thus theorems about odd primitive non-deficient numbers are also theorems about odd perfect numbers. Importantly, while we do not know if there are any odd perfect numbers, there are many odd primitive non-deficient numbers, with $945$ being the smallest example. 

Inequality \eqref{basic h inequality} also leads to the observation that no perfect number can be divisible by any other perfect or abundant number. This observation is frequently used to restrict what the prime divisors of an odd perfect number can look like. For example, it is not hard to show that if $n$ is an odd perfect number then $105 \nmid n$, and similarly that $(3)(5)(11)(13)\nmid n$. 

 One of the basic properties of $h(n)$ is that $h(n) \leq H(n)$ with equality if and only if $n=1$. In fact, $H(n)$ is also the best possible function which bounds $h(n)$ from above and which only keeps track of which primes divide $n$ but not of their multiplicity. In particular,

$$\lim_{k \rightarrow \infty} h(n^k) =H(n).$$

When $n$ is perfect, one has $h(n)=2$, and thus 
$H(n)> 2$. When $n$ is perfect, $H(n)$ is a natural function to estimate. 

Prior work on estimating $T(n)$ and $H(n)$ when $n$ is an odd perfect goes back to a 1958 paper of Perisastri \cite{Perisastri} who showed that if $n$ is an odd perfect number, then 
$$ \frac{1}{2} < T(n) < 2\log \frac{\pi}{2}.$$

Subsequently, tighter bounds were proven in a series of papers by Cohen, Hagis, and Suryanarayana \cite{Cohen 1978,Cohen 1980, Suryanarayana II, Suryanarayana III, Suryanarayana and Hagis}, which included the result that $T(n) < \log 2$.  

We summarize their results in the Proposition below.

\begin{proposition}(Cohen, Hagis, and Suryanarayana \cite{Cohen 1978,Cohen 1980, Suryanarayana III, Suryanarayana and Hagis}) If $n$ is an odd perfect number, then $T(n)$ is bounded as given in the table below, broken down by the greatest common divisor of $n$ and $15$ being as listed.
\begin{center}
\begin{tabular}{ |c|c|c| } 
 \hline
 $\gcd(15,n)$ & Lower bound for $T(n)$ & Upper bound for $T(n)$ \\ 
\hline
$1$  &   $0.667450$ & $0.693148$  \\
 $3$ &    $0.603831$ & $0.657304$  \\
 $5$  &    $ 0.647387$ & $0.678036$  \\ 
 $15$  &   $0.596063$ & $0.673770$   \\
 \hline
\end{tabular}
\end{center}
\end{proposition}.

Note that the $0.693148$ in the case where $\gcd(15,n)=1$ is due to the aforementioned bound that $T(n) < \log 2$.   

Similarly, bounds on $H(n)$ were proved by Suryanarayana \cite{Suryanarayana III}. These bounds can also broken down into the same major cases.

\begin{proposition}(Suryanarayana \cite{Suryanarayana III}) If $n$ is an odd perfect number, then $H(n)$ is bounded according to the greatest common divisor of $n$ and 15 as follows. \label{Sur H upper bound}
\begin{center}
\begin{tabular}{ |c|c| } 
 \hline
 $\gcd(15,n)$ & Upper bound for $H(n)$ \\ 
\hline
$1$  &  $2.014754$  \\
 $3$ &    $2.096234$  \\
 $5$  &     $2.031002$  \\ 
 $15$  &   $2.165439$   \\
 \hline
\end{tabular}
\end{center}
\end{proposition}

The upper bounds in the table in Proposition \ref{Sur H upper bound} are rounded up from the actually proven quantities, all of which are rational multiples of $\zeta(3)$ where $\zeta(s)$ is the Riemann zeta function.

Throughout the rest of this paper, unless stated otherwise, we will assume that n is a positive integer and its canonical factorization is given as t  $n=p_1^{a_1} p_2^{a_2} \cdots p_k^{a_k}$ where the $p_i$ are primes where  $p_1 < \cdots < p_k$.

When $n$ is non-deficient, the surplus $S(n)$ of $n$ is defined by by $S(n)=H(n)-2$. 
The surplus played a major role in related prior work by the author\cite{Zelinsky component compared to radical}. Using this notation, Lemma 6 from that paper becomes:

\begin{lemma}\label{Generalized Puchta lemma}(Zelinsky) Let $n$ be a primitive non-deficient number with $n=p_1^{a_1} \cdots p_k^{a_k}$ with primes $p_1 <  \ldots < p_k$. Let $R=\mathrm{rad}(n) = p_1p_2\cdots p_k$. Assume further that $$S(n) \geq  \alpha$$ for some positive $\alpha <1$ Then there exists an $i$, $1 \leq i \leq k$ such that 
$$p_i^{a_i+1} < \max\{\frac{2\left(k+2+p_1\right)}{\alpha },k(k+1)\}.$$
\end{lemma}

The main results of this paper are the following two propositions.

\begin{proposition} Assume $n$ is a primitive non-deficient number with prime factorization given by $n=p_1^{a_1}p_2^{a_2} \cdots p_k^{a_k}$ with primes $p_1$, $p_2$, $\ldots p_k$ and $p_1 <  \ldots < p_k$, and surplus $S=S(n)$.  Then $$T(n) \geq \log 2 -\frac{25}{64p_1} + \frac{S}{2} -\frac{S^2}{4}.$$ \label{first lower upper bound for T}
\end{proposition} 
We also have the following upper bound for $T(n)$.
\begin{proposition} Assume $n$ is a primitive non-deficient number with $n=p_1^{a_1}p_2^{a_2} \cdots p_k^{a_k}$, and $p_1 <  \cdots < p_k$. Assume that the $a_i$ are even for all $i$ except some $t$ where $a_t$ is odd. Assume further that $\frac{p_t+2}{2} \geq p_1 \geq 11$. Then $$T(n) < \log 2 - \frac{11}{50p_1^2}.$$ \label{upper bound for T, between all prim defs and opns}
\end{proposition}

Proposition \ref{upper bound for T, between all prim defs and opns} will give a bound for $T(n)$ for $n$ when $n$ is an odd perfect number since any odd perfect number $n$ will satisfy the hypotheses of the proposition.  We will also prove slightly stronger but more complicated bounds, as well as bounds for $H(n)$. In particular, we prove the following.

\begin{proposition}
    If $n$ is an odd perfect number, then $$H(n) \leq 2 + \frac{9}{4p_1^2}.$$ \label{upper bound for H for an OPN}    
\end{proposition}

Propositions of these sorts are of interest since if one had tight enough upper and lower bounds for $T(n)$ and $H(n)$, one would be able to prove results of the form  ``Any odd perfect number is divisible by a prime which is at most $K$'' for some constant $K$. However, while this long-range goal justifies interest in these bounds, it must be noted that it does not seem like the techniques used in this paper can be used to obtain tight enough bounds to achieve that goal without major new insights.

If $n$ is an odd perfect number, then other types of bounds relating $p_1$ to $n$ also exist. The most basic and classic result in this regard is due to Servais \cite{Servais} who proved the bound for odd perfect numbers, but the same argument applies to any non-deficient number. 

\begin{proposition}(Servais) 
Assume that $n$ is a non-deficient number with  $n=p_1^{a_1} p_2^{a_2} \cdots p_k^{a_k}$ where the $p_i$ are primes with $p_1  < \cdots < p_k$. Then $p_1 \leq k$.
\end{proposition}
\begin{proof}  Assume that $n$ is non-deficient. Thus, $$2 \leq h(n) < H(n) = \prod_{i=1}^k \frac{p_i}{p_i-1}$$

 Since $\frac{x}{x-1}$ is a decreasing function for positive $x$ and the $i$th prime after $p_1$ is at least $p_1+i$, we must have
$$2 <  \left(\frac{p_1}{p_1-1}\right)\left(\frac{p_1+1}{p_1} \right)\left(\frac{p_1+2}{p_1+1}\right) \cdots \left(\frac{p_1+k-1}{p+k-2}\right) = \frac{p_1+k-1}{p_1-1}.  $$
Thus $$2p_1-2 < k -1,$$
and so $k > p_1-1$, and so $k \geq p_1$.
\end{proof}

Gr\"un \cite{Grun} observed that one can tighten this bound by taking into account that every prime other than 2 is odd. This  yields a tighter but still linear bound. 

Norton \cite{Norton} used the prime number theorem to give a tighter bound for $k$ in terms of $p_1$  as well as bounds on the size of $n$ in terms of $p_1$. These were slightly tightened by the author \cite{Zelinskybig}. The bounds of Norton and the author are  asymptotically best possible when restricted to non-deficient numbers, but there is likely much room for improvement if one is restricted to odd perfect numbers. 

Similar bounds relating $k$ to $p_i$ for small $i$ also exist. Kishore \cite{Kishore} proved bounds similar to Gr\"un relating small prime factors to $k$. Kishore's bounds are linear bounds on $k$ in terms of the size of the first few smallest prime factors. The author \cite{Zelinskyfollowup} then adopted Norton's argument for small prime factors in a way that again is essentially asymptotically best possible for non-deficient numbers. Similar bounds have also recently been proven by Aslaksen and Kirfel \cite{AK1}.

In a different direction, bounds relating $n$ and its large prime factors also exist.  Acquaah and Konyagin \cite{AK} showed that if $n$ is an odd perfect number then 
\begin{equation}
\label{p_k AK bound} p_k^3 < 3n.
\end{equation}

The proof for the bound of Acquaah and Konyagin carefully uses that $n$ is perfect, and therefore the proof does not apply to all primitive non-deficient numbers. One might wonder if Acquaah and Konyagin's bound can be extended to apply to all primitive non-deficient numbers. The answer is no.  Consider for example $n=9765$ which is an odd primitive non-deficient number. The largest prime factor of $9765$ is 31, and $(3(9765))^\frac{1}{3} = 30.826 \cdots $.  It is not hard to show that if $n$ is primitive non-deficient number then one must have $p_k^2 < 2n$. This weaker bound is essentially best possible.  However, empirically, the vast majority of odd primitive deficient numbers satisfy Inequality \eqref{p_k AK bound}. Acquaah and Konyagin's bound is notable also for being strong enough to rule out some ``spoof'' perfect numbers, especially that of Descartes's example. 

Descartes noted that $D=198585576189$ looks almost like an odd perfect number. In  particular, one may factor it as $D = 3^2 7^2 11^2 13^2 22021$. One has then that 
$$\sigma(D) = (3^2+3+1)(7^2+7+1)(11^2+11+1)(13^2 +13+1)(22021+1)= 2D$$
where we ignore that 22021 is in fact not a prime number. Examples similar to Descartes's example were extensively studied in \cite{Pace spoof group} which greatly generalized this notion as follows. Given an integer $n$, define a {\emph{factorization}} of $n$ to be an expression of the form  $$n = \prod_{i=1}^{k} x_i^{b_i},$$ where the $x_i$ are integers and the $b_i$ are positive integers. Notice that a factorization can also be thought of as a multiset of ordered pairs of the form $(x_i,b_i)$.  

Define a function $\tilde{\sigma}$ on the multiset of such ordered pairs as follows:
$$
\tilde{\sigma}\Big(\{(x_i,b_i):1\leq i\leq k\}\Big) =\prod_{i=1}^{k}\left(\sum_{j=0}^{b_i}x_i^{j} \right).
$$

A factorization as above is \emph{spoof perfect} if $\tilde{\sigma}(\prod_{i=1}^{k}x_i^{b_i})=2\prod_{i=1}^{k}x_i^{b_i}$.  If $n=p_1^{a_1}p_2^{a_2}\cdots p_k^{a_k}$ with the $p_i$ distinct primes, then $$\tilde{\sigma}((p_1,a_1),(p_2,a_2) \cdots (p_k,a_k))= \sigma(n).$$ Thus $\tilde{\sigma}$ provides a generalization of the classical $\sigma$ function. 

By the above remark, any actual perfect number $n$ gives rise to a spoof perfect factorization. 

In this context, Descartes number $D$ gives rise to a spoof factorization given by the ordered pairs $\{(3,2),(7,2),(11,2), (13,2),(22021,1)\}$. The number $D$ with this factorization is particularly notable for being the only known example where all the $x_i$ are positive and odd. However, many other examples exist where one is allowed to have negative values for $x_i$. 

The existence of Descartes's number, along with similar examples, place heavy restrictions on the effectiveness of different sorts of arguments about odd perfect numbers; any restrictions which apply to spoofs cannot show that no odd perfect numbers exist without also showing that spoofs do not exist. Many results about odd perfect numbers do not avoid this issue. However, Acquaah and Konyagin's bound, which carefully uses that the largest prime factor is actually a prime, manages to be tighter than the best possible bound that applies to spoofs. In particular, 
$(3(198585576189))^\frac{1}{3} = 8414.396 \ldots < 22021$.

Bounds similar to that of Inequality \eqref{p_k AK bound} have been proven for the second and third largest prime factors of an odd perfect number $n$. In particular, the author \cite{ZelinskySecond} proved that 

\begin{equation} p_{k-1} < (2n)^{1/5} \label{Zelinskysecond},  
\end{equation}

and that 

\begin{equation}\label{bc inequality from previous paper}  p_kp_{k-1} < 6^{1/4}n^{1/2}. \end{equation}
Note that the bound in Equation \eqref{Zelinskysecond}, unlike Inequality \eqref{p_k AK bound} is not so strong as to avoid the Descartes example. In particular, $(2(198585576189))^{\frac{1}{5}}  = 208.831 \cdots > 13.$

Bibby, Vyncke, and the author \cite{BVZ} proved a similar bound for $p_{k-3}$ as well as a bound on $p_kp_{k-1}p_{k-2}$. That paper also includes a weak bound of a similar sort on $p_{k-i}$ for any $i$. 
 
A major role in most work on this topic is played by Euler's theorem on odd perfect numbers.

\begin{lemma} (Euler) \label{Euler's theorem} If $n$ is an odd perfect number, then $n=q^e m^2$
 where $q$ is prime, $\gcd(q,m)=1$, and where $q \equiv e \equiv 1 \pmod 4. $   
\end{lemma}

When we have an odd perfect number $n$ in the above form, we will refer to $q$ as the special prime. Note that Euler's theorem is in modern terms an easy exercise in modular arithmetic, but its ease reflects that it is a very weak statement. The same restriction applies just as well to any odd $n$ where one only assumes that $\sigma(n) \equiv 2 \pmod 4 $.

\section{Proofs of the main bounds}

Before we prove our main results, we need a few technical lemmata. 

\begin{lemma} If $0 \leq x \leq \frac{1}{11}$, then $$1+x+x^2 \geq e^{x+\frac{2x^2}{5}},$$ \label{x lemma to e form with 11} with equality if and only if $x=0$.
\end{lemma}
\begin{proof}
 We note that this is true for $x=0$, and so we may assume that $x \neq 0$.  Proving this Lemma is the same as showing that in the desired range  that
$$e^{x+\frac{2x^2}{5}} -(1+x+x^2) \leq 0.$$ 
Expanding this quantity using the Taylor series for $e^x$ and we obtain
\begin{equation}
\begin{split}
    &e^{x+\frac{x^2}{3}} -(1+x+x^2)  =\\ & \frac{-1}{10}x^2 + \frac{2}{5}x^3 + \frac{2}{25}x^4 + \frac{(x+\frac{2x^2}{5})^3}{6} + \frac{(x+\frac{2x^2}{5})^4}{24} + \frac{(x+\frac{2x^2}{5})^5}{120} + \cdots .\label{difference expansion for 11 version}
\end{split}
\end{equation}

Thus, the result can be proved if we prove that if $x \leq \frac{1}{11}$, then

\begin{equation}   \frac{2}{5}x + \frac{2}{25}x^2 + \frac{(x+\frac{2x^2}{5})^3}{6x^2} +  \frac{(x+\frac{2x^2}{5})^4}{24x^2} + \frac{(x+\frac{2x^2}{5})^5}{120x^2} + \cdots < \frac{1}{10}. \label{difference expansion after dividing by x squared and moving 1/10 over} .\end{equation}

We will break up the right-hand side of Equation \eqref{difference expansion after dividing by x squared and moving 1/10 over} and estimate its first three terms, and then estimate the remaining terms.

If $x \leq \frac{1}{11}$, then

 $$\frac{2}{5}x + \frac{2}{25}x^2 + \frac{(x+\frac{2x^2}{5})^3}{6x^2}= \frac{2}{5}x + \frac{2}{25}x^2 + \frac{(1+\frac{2x}{5})^2(x+\frac{2x^2}{5})}{6} < 0.053891.$$

We need to also estimate $$\frac{(x+\frac{2x^2}{5})^4}{24x^2} + \frac{(x+\frac{2x^2}{5})^5}{120x^2} + \cdots.$$

We have:
\begin{equation}
\begin{split}
     \frac{(x+\frac{2x^2}{5})^4}{24x^2} + \frac{(x+\frac{2x^2}{5})^5}{120x^2} + \cdots &\leq \frac{(1+\frac{2x}{5})^2(x+\frac{2x^2}{5})^2}{24} \sum_{i=0}^\infty \left(\frac{(x+\frac{2x^2}{5})}{5}\right)^i  \\
    & \leq \frac{1}{20} \sum_{i=0}^\infty \left(\frac{(\frac{1}{11}+\frac{2(\frac{1}{11})^2}{5})}{5}\right)^i \\
    & \leq 0.0459663.
\end{split}
\end{equation}
Now, note that $0.053891+0.0459663 = 0.0998573 < \frac{1}{10}$ which proves the result. \end{proof}
Note that Lemma \ref{x lemma to e form with 11} is nearly identical to a similar Lemma in Cohen \cite{Cohen 1978} but the conclusion is tighter, though at the cost of being true for a smaller range of values of $x$.

\begin{lemma} If $p \geq 19$ is a prime number, then $\pi(p) \geq \frac{p}{\log p}+1$. \label{RS lower bound for pi}
\end{lemma}
\begin{proof} We have  from Rosser and Schoenfeld's classic bounds on the prime number theorem \cite{Rosser Schoenfeld}, that if $x$ is a real number and $x \geq 59$ then $$\pi(x) \geq \frac{x}{\log x}\left(1+\frac{1}{2\log x}\right).$$ This implies that when $p$ is prime and $p \geq 59$, then $\pi(p) \geq \frac{p}{\log p}+1$. We need to then only check that the inequality is true for the primes which are between 19 and 53. 
\end{proof}

We also need the following bound from Rosser and Schoenfeld \cite{Rosser Schoenfeld}. 
\begin{lemma}  Let $P_j$ be the $j$th prime number. If $j \geq 1$, then $P_j \geq j \log j$, where $P_j$ is the $n$th prime number. \label{Rosser-Schoenfeld lower bond for prime}
\end{lemma}

\begin{lemma} Let $q$ be a prime with $q \geq 3$ . Let $M(q)=\sum_{p \geq q} \frac{1}{p^2}$ where $p$ ranges over all primes at least $q$. Then $M(q) < \frac{1}{q\log q}\left( 1 + \frac{2 \log \log q}{\log q}\right).$ \label{1/p^2 sum lemma}
\end{lemma}
\begin{proof} Let $q$ be a prime which is at least 19. Assume that $M(q)=\sum_{p \geq q} \frac{1}{p^2}$. We have, using Lemma \ref{RS lower bound for pi} and Lemma \ref{Rosser-Schoenfeld lower bond for prime} that \begin{equation}\begin{split} M(q) \leq \sum_{n \geq \pi(q)}\frac{1}{(n \log n)^2} & \leq \int_{\frac{q}{\log q}}^{\infty} \frac{1}{(x \log x)^2} dx\\ &<  \int_{\frac{q}{\log q}}^{\infty} \frac{2 +\log x }{(x \log x)^3} dx = \frac{1}{\frac{q}{\log q} \log \frac{q}{\log q}}.\end{split}  \end{equation}

\noindent It is now straightforward to verify that if $q \geq 19$ then $$\frac{1}{\frac{q}{\log q} \log \frac{q}{\log q}} <   \frac{1}{q\log q}\left( 1 + \frac{2 \log \log q}{\log q}\right), $$ and so the result is proven for $q \geq 19$. Using the inequality for $q \geq 19$, one can verify directly   that the same inequality holds for $q=3$, $5$, $7,$ $11$, and $13$.
\end{proof}

Note that Lemma \ref{1/p^2 sum lemma} is asymptotically best possible. However, we suspect that if $q$ is sufficiently large then one in fact has $M(q) \leq \frac{1}{q\log q}$.

\begin{lemma} Let $q$ be an odd prime. Set $E=\prod_{p \geq q}\frac{p^2}{p^2-1}$. Then 
$$\log E \leq \frac{1}{q\log q}\left( 1 + \frac{2 \log \log q}{\log q}\right) + \frac{4}{q^3}.$$
\end{lemma}
\begin{proof}
    We have $$\log E = \sum_{p \geq q} \log \left(1+ \frac{1}{p^2-1}\right)\leq \sum_{p \geq q} \frac{1}{p^2-1} = \sum_{p \geq q}\left( \frac{1}{p^2} + \frac{1}{p^2(p^2-1)}\right).  $$

    We now note that $\sum_{p \geq q} \frac{1}{p^2(p^2-1)}  \leq \frac{4}{q^3}$, and combine that with Lemma \ref{1/p^2 sum lemma}.
\end{proof}

\begin{lemma} If $x$ is an integer and  $x \geq 2$, then $\log(1+\frac{1}{x-1}) \leq \frac{1}{x}+ \frac{25}{32x^2}$. \label{25/32 x squared lemma}
\end{lemma}
\begin{proof}
One may verify directly that this is satisfied for $x=2$, $x=3$ and $x=4$. So  we may assume that $x \geq 5$. We have from the Taylor series for $\log (1+t)$, that if $0 \leq t <1$ that

\begin{equation} \log (1+t) \leq t - \frac{t^2}{2} + \frac{t^3}{3} \label{truncated log}. 
\end{equation}

We can substitute $t=\frac{1}{x-1}$ into Inequality \eqref{truncated log} to obtain:

\begin{equation}\log(1+\frac{1}{x-1}) \leq \frac{1}{x-1}  - \frac{1}{2(x-1)^2} + \frac{1}{3(x-1)^3} \label{truncated log 2}. \end{equation}

We have then

\begin{equation}\frac{1}{x-1}  - \frac{1}{2(x-1)^2} + \frac{1}{3(x-1)^3} = \frac{1}{x} + \frac{1}{x(x-1)} - \frac{1}{2(x-1)^2} + \frac{1}{3(x-1)^3} .\label{truncated log 3} \end{equation}

Thus, the result holds if we show that 

$$ \frac{1}{x(x-1)} - \frac{1}{2(x-1)^2} + \frac{1}{3(x-1)^3} \leq \frac{25}{32x^2}.$$

We have that $$ \frac{1}{x(x-1)} - \frac{1}{2(x-1)^2} + \frac{1}{3(x-1)^3} = \frac{3x^2 -7x+6}{6x(x-1)^3 } $$
Since $x \geq \frac{3}{2} $, we have that

$$\frac{3x^2 -7x+6}{6x(x-1)^3 } \leq \frac{3x^2 -3x}{6x(x-1)^3 } = \frac{1}{2(x-1)^2}.$$
Since $x \geq 5$, $x-1 \geq \frac{4}{5}x$, and so 
$$\frac{1}{2(x-1)^2} \leq \frac{25}{32x^2}.$$
\end{proof}

We now prove Proposition \ref{first lower upper bound for T}.
\begin{proof} Assume that $n$ is a non-deficient number with factorization as above. Then
\begin{equation}2 +S=H(n) = \prod_{i=1}^k \frac{p_i}{p_i-1} . \label{2+ S} \end{equation}

We take the logarithm of both sides of Equation \eqref{2+ S} to obtain that 
\begin{equation} \log (2+S) = \sum_{i=1}^k \log \frac{p_i}{p_i-1} = \sum_{i=1}^k \log \left(1+\frac{1}{p_i-1}\right). \end{equation}
We then apply Lemma \ref{25/32 x squared lemma} to obtain that 

\begin{equation}\begin{split} \log (2+S) &<  \sum_{i=1}^k\left( \frac{1}{p_i} + \frac{25}{32p_i^2}\right) = \sum_{i=1}^k \frac{1}{p_i} + \frac{25}{32}\sum_{i=1}^k \frac{1}{p_i^2}\\ &   <  \sum_{i=1}^k \frac{1}{p_i} + \frac{25}{64p_1}.\end{split}
\end{equation}

In the last step of the above inequality, we are using that the primes are odd to get an upper bound on $\sum_{i=1}^\infty \frac{1}{p_i^2} $.

Since $S<1$ we have $\log (2+S) \geq \log 2 + \frac{S}{2} - \frac{S^2}{4}$, which implies the desired result.
\end{proof}

Proposition \ref{first lower upper bound for T} helps illustrate more explicitly why prior results cannot progress beyond $\log 2$ as their general upper bound.

Similarly, we may use Lemma \ref{1/p^2 sum lemma} to obtain:

\begin{proposition} Assume $n$ is a primitive non-deficient number with prime factorization given by $n=p_1^{a_1}p_2^{a_2} \cdots p_k^{a_k}$ with primes $p_1$, $p_2$, $\cdots p_k$ and $p_1 <  \ldots < p_k$, and surplus $S=S(n)$.  Then $$T(n) \geq \log 2 - \frac{25}{32p_1\log p_1}\left( 1 + \frac{2 \log \log p_1}{\log p_1}\right) + \frac{S}{2} -\frac{S^2}{4}.$$ \label{second lower upper bound for T}\end{proposition}
We now prove Proposition \ref{upper bound for T, between all prim defs and opns}.

\begin{proof} The proof given here is essentially a tightening of the proof given in Cohen \cite{Cohen 1978} to show that $T(n) < \log 2$. There are three changes.  First, we use Lemma \ref{x lemma to e form with 11} rather than Cohen's Lemma. Second, we keep more careful track of the second order terms.  Third, we extend the result to primitive non-deficient numbers of a specific form which includes odd perfect numbers, rather than just prove the result for odd perfect numbers. 

Assume as given. Then we have:
\begin{equation}
    2=\frac{\sigma(n)}{n} = \prod_{1 \leq i \leq k} \left(1+\frac{1}{p_i} + \cdots +\frac{1}{p_i^{a_i}} \right)\geq  \left(1+\frac{1}{p_t}\right)\prod_{\substack{1 \leq i \leq k\\ i \neq t}} \left(1+\frac{1}{p_i}+\frac{1}{p_i^2}\right). \label{Starting inequality for T upper}
\end{equation}

Since every $p_i$ is assumed to be at least 11, we may apply Lemma \ref{x lemma to e form with 11} to Equation \eqref{Starting inequality for T upper} to obtain that

\begin{equation} 2 \geq (1+\frac{1}{p_t})\prod_{\substack{1 \leq i \leq k\\ i \neq t}} e^{\frac{1}{p_i} + \frac{2}{5p_i^2}},
\end{equation}

and hence

\begin{equation} \log 2 \geq \log\left(1+\frac{1}{p_t}\right) + \sum_{\substack{1 \leq i \leq k\\ i \neq t}} \frac{1}{p_i} + \frac{2}{5p_i^2}. \label{log 2 upper next} \end{equation}

Since $\log (1+x) > x -\frac{x^2}{2}$, we obtain from Equation \eqref{log 2 upper next} that

\begin{equation} \log 2 > \sum_{i=1}^k \frac{1}{p_i} + \frac{2}{5}\sum_{\substack{1 \leq i \leq k\\ i \neq t}}  \frac{1}{p_i^2} - \frac{1}{2p_t^2}. \label{penultimate log}
\end{equation}

As we have $2p_1-1 \leq p_t$, and $p_1 \geq 3$, we have that $\frac{5}{3}p_1 \leq p_t$. Thus,
$\frac{9}{50}\frac{1}{p_1^2} \geq \frac{1}{2p_t^2}$, which combines with Equation \eqref{penultimate log} to obtain the desired bound. 
\end{proof}

We obtain as a corollary:
\begin{corollary}
    
 If $n$ is an odd perfect number with smallest prime factor $p_1 \geq 11$, then $T(n) < \log 2 - \frac{11}{50p_1^2}.$ \label{OPN upper bound for T(n)}   
\end{corollary}
\begin{proof} To prove this we just need to show that any odd perfect number satisfies the properties of Proposition \ref{upper bound for T, between all prim defs and opns}.  But this is essentially just Euler's theorem for odd perfect numbers, along with the fact that if $p_t$ is the special prime of $n$, then $$p_t+1\mid\sigma(p_t^{a_t}),$$ and hence $p_t+1\mid2n$. Thus, $n$ must be divisible any odd prime factor of  $p_t+1$, and hence $p_1 \leq \frac{p_t+1}{2}$.
\end{proof}

The main usage of Proposition \ref{upper bound for T, between all prim defs and opns} was in the case that our $n$ was an odd perfect number, and those likely do not exist. Thus,  it is worth asking if there are any numbers which actually satisfy the hypotheses of Proposition \ref{upper bound for T, between all prim defs and opns}. In fact, such numbers do exist. An example is $n=(3^2)(5)(11^2)(13^2)(17^2)(19^2)(23^2)$. We strongly suspect that there are infinitely many such numbers, but do not see how to prove that.

An easy consequence of Corollary \ref{OPN upper bound for T(n)} and Proposition \ref{second lower upper bound for T}
is the following upper bound on $S(n)$.

\begin{corollary} Assume $n$ is an odd perfect number with $p_1 \geq 11$. Then $S(n) < \frac{25}{8p_1 \log p_1}$. 
\end{corollary}
\begin{proof} Assume as given. From Corollary \ref{OPN upper bound for T(n)} and Proposition \ref{second lower upper bound for T}
we have  \begin{equation} \frac{25}{16p_1 \log p_1} \geq \frac{S(n)}{2} - \frac{S(n)^2}{4} + \frac{11}{50p_1^2}.   \label{Combo props 1}\end{equation}    

From earlier mentioned results on $S(n)$, we have $S(n) < \frac{1}{5}$, and so $\frac{S(n)^2}{4} \leq \frac{S(n)}{20}.$ Thus, we have from Equation \eqref{Combo props 1} that 
$$ \frac{25}{16p_1 \log p_1} \geq \frac{9}{20}S(n),$$
which implies that

  \begin{equation} S(n) \leq \frac{125}{72p_1 \log p_1}. \label{preliminary S upper bound} \end{equation}
Equation \eqref{preliminary S upper bound} implies that \begin{equation} \frac{S(n)^2}{4} \leq \left(\frac{125}{72p_1 \log p_1}\right)^2 < \frac{11}{50p_1^2}. \label{SS/4 is less than p2 term}\end{equation}

The last step in the chain of inequalities in Equation \eqref{SS/4 is less than p2 term} is using that $p_1 \geq 11$ and then estimating $\log 11$.

Equation \eqref{SS/4 is less than p2 term} combines with Equation \eqref{Combo props 1} to get that

$$ \frac{25}{16p_1 \log p_1} \geq \frac{S(n)}{2}, $$ which implies the claimed bound. \end{proof}

We can tighten this somewhat using the following method.  The idea of the following Lemma is implicitly used in Suryanarayana's proof \cite{Suryanarayana III} but we are making the general form explicit here.

\begin{lemma} Let $n$ be a positive integer of the form $n=p_1^{a_1}p_2^{a_2} \cdots p_k^{a_k}$  where $p_1, p_2, \cdots p_k$ are primes with $3 \leq p_1 < \cdots <  p_k$, and where all $a_i$ are even for $1 \leq i \leq k$ except for a single $a_j$ which is odd. Then $$\frac{H(n)}{h(n)} \leq \frac{p_j^2}{p_j^2 -1}\prod_{1 \leq i \leq k, i \neq j} \frac{p^3}{p^3 -1}. $$ \label{zeta 3 lemma form}
\end{lemma}
See Suryanarayana \cite{Suryanarayana III} for the proof. 

We are now in a position where we can progress on our tighter estimate for $H(n)$.

\begin{lemma} Let $n$ be a positive integer of the form $n=p_1^{a_1}p_2^{a_2} \cdots p_k^{a_k}$  where $p_1, p_2, \ldots p_k$ are primes with $3 \leq p_1 < \ldots <  p_k$, and where all $a_i$ are even for $1 \leq i \leq k$ except for a single $a_j$ which is odd. Then
 $$H(n) \leq h(n)\left(1  +\frac{3}{4p_1^2}\right). $$  \label{Prelude to H bound}
\end{lemma}
\begin{proof}
    Assume as given. We have via logic similar to Lemma \ref{zeta 3 lemma form}, that

    \begin{equation}
      H(n) \leq h(n)\left(\frac{p_j^2}{p_j^2 -1}\right) \prod_{\substack{1 \leq i \leq k\\ i \neq j}} \frac{p_i^3}{p_i^3-1}, 
    \end{equation}

    which is equivalent to 

\begin{equation}
      \log H(n) \leq \log h(n) + \log \left(\frac{p_j^2}{p_j^2 -1}\right) +\sum_{\substack{1 \leq i \leq k\\ i \neq j}} \log \frac{p_i^3}{p_i^3-1} \label{log form of H upper bound in term of h}.
    \end{equation}

Since $x \geq  \log (1+x)$, we have from Equation \eqref{log form of H upper bound in term of h}

\begin{equation} \log H(n) \leq \log h(n) + \frac{1}{p_j^2 -1} + \sum_{\substack{1 \leq i \leq k\\ i \neq j}} \\frac{1}{p_i^3-1}. \label{Surplus bound eq 3}
\end{equation}

Equation \eqref{Surplus bound eq 3} implies that 

\begin{equation} \log H(n) \leq \log h(n) + \frac{1}{p_j^2} + \sum_{\substack{1 \leq i \leq k\\ i \neq j}} \frac{1}{p_i^3}. \label{Surplus bound eq 4}
\end{equation}

It is not hard to see that $$ \frac{1}{p_j^2} + \sum_{1 \leq i \leq k, i \neq j} \frac{1}{p_i^3} \leq \frac{3}{4p_1^2},$$ which combines with Inequality (ref{Surplus bound eq 4}) to imply the desired result. 
The remainder of the proof involves an estimate essentially identical in form to that made in the proof of Proposition \ref{upper bound for T, between all prim defs and opns}.
\end{proof}

As an corollary to Lemma \ref{Prelude to H bound} we obtain  an immediate proof of Proposition \ref{upper bound for H for an OPN}. Note that Proposition  \ref{upper bound for H for an OPN} is better than the bounds from Suryanarayana when $p \geq 13$. The analogous bound for  Proposition \ref{upper bound for H for an OPN} is satisfied by Descartes number. However, this is likely not satisfied by all of the most general forms of spoofs in \cite{Pace spoof group}. In particular,  since one of the forms of spoofing allowed there is to have repeated copies of a prime factor be treated as distinct primes. But the proof for Proposition \ref{upper bound for H for an OPN} uses bounds on series involving the the sums of primes and estimating those sums. 

An argument similar to the proof of Proposition \ref{upper bound for H for an OPN} can be used to show the following:

\begin{proposition} There is a constant $C$ such that if $n$ is a primitive non-deficient number with smallest prime factor $p_1$, then $$H(n) \leq 2 + \frac{C}{p_1}.$$
\end{proposition}

\section{Related questions}

The following seems like a basic question but as far as we are aware is not in the literature. Let $X$ be the set of real numbers defined as follows: The set $X$ consists of those real numbers $x$, such that there exist infinitely many odd primitive non-deficient $n$ with $H(n) \leq x$. What is the supremum of $X$?

A major difficulty in improving the bounds  for an odd perfect number for both $S(n)$ and $T(n)$ is the behavior of the special prime. In particular, there is difficulty in the case where special prime is raised to just the first power. There are a variety of modulo restrictions regarding when prime $q$ can be a special prime. We mentioned earlier that one must have $(3)(5)(7) \nmid n$ for an odd perfect number $n$, and thus one cannot have any special prime $q$ where $q \equiv -1 \pmod{105}$. But other similar restrictions are known. 

We also have the following slightly more subtle folklore result which as far as the author is aware is not explicitly in the literature. 

\begin{proposition} If $q \equiv -1 \pmod{165}$, then $q$ is not the special prime for any odd perfect number $q$.\label{165 Prop}
\end{proposition}

Proposition \ref{165 Prop} has a curious nature. If one has not seen it before, one might try to prove that no odd perfect number can be divisible by 165 from which the result would then follow immediately. However, any attempt to prove that 165 cannot divide $n$ using current techniques fails. The key to proving Proposition \ref{165 Prop} is to instead prove the following, weaker statement.

\begin{lemma} Let $n$ be an odd perfect number. If $165\mid n$, then one must have $5\mid\mid n$.\label{3 5 11 implies 5 special lemma}
\end{lemma}
\begin{proof} Assume that $n$ is an odd perfect number with $165\mid n$. If $3^2 \mid n$, then $\sigma(3^2) \mid n$, which would contradict the fact that one cannot have $(3)(5)(11)(13) \mid n$. Thus, we must have $3^4 \mid n$. Now, $5^2 \mid n$, then one must have $3^4 5^2 11^2 \mid n$, and $3^4 5^2 11^2$ is abundant.  Thus we must have $5 \mid\mid n$ which proves the desired result. 
\end{proof}

Proposition \ref{165 Prop}  follows from Lemma \ref{3 5 11 implies 5 special lemma} since an odd perfect number must have exactly one special prime factor $q$. If $q$ is the special prime for $n$, and $q \equiv -1 \pmod{165} $, then $5$ is not the special prime but $165\mid n$. This contradicts Lemma \ref{3 5 11 implies 5 special lemma}.

A natural question in this context is if a prime $q$ can be ruled out as being the special prime where the contribution from $q$ itself is used in the method above. A slightly weaker property is for there to exist a prime $q$ such that $q \equiv 1 \pmod 4$, where
$H(\frac{q+1}{2}) < 2$ and $H(q\frac{q+1}{2}) > 2$. We will show that even this weaker property cannot occur. We will need a Lemma first.

\begin{lemma} Suppose that $p_1, p_2, \ldots p_m$ are distinct odd primes, where  $$\frac{p_1}{p_1-1} \cdots \frac{p_m}{p_m-1} \leq 2,$$ and $q$ is an odd prime such that $$\left(\frac{p_1}{p_1-1} \cdots \frac{p_m}{p_m-1}\right)\frac{q}{q-1} > 2.$$  Then $q \leq 2(p_1-1)(p_2-1) \cdots (p_m-1)-1$. \label{precursor to no tricky specials}
\end{lemma}
\begin{proof}
Assume as given. Since $\frac{p_1}{p_1-1} \cdots \frac{p_m}{p_m-1}$ reduces to some fraction with an odd numerator, this product cannot be equal to 2. Furthermore, the denominator of this fraction is at most $(p_1-1)(p_2-1) \cdots (p_m-1)$. Thus we have
 \begin{equation}
     \frac{p_1}{p_1-1} \cdots \frac{p_m}{p_m-1} \leq 2 - \frac{1}{(p_1-1)(p_2-1) \cdots (p_m-1)}. \label{Q equation 1}
 \end{equation}

By similar logic, we have \begin{equation}
     \frac{p_1}{p_1-1} \cdots \frac{p_m}{p_m-1}\frac{q}{q-1} \geq 2 + \frac{1}{(p_1-1)(p_2-1) \cdots (p_m-1)(q-1)}.\label{Q equation 2}
 \end{equation}

 Combining Equation \eqref{Q equation 2} and Equation \eqref{Q equation 1}, we obtain that

 \begin{equation}\left(2 - \frac{1}{(p_1-1) \cdots (p_m-1)}\right)\frac{q}{q-1} \geq 2 + \frac{1}{(p_1-1)\cdots (p_m-1)(q-1)},  \end{equation}
 
which is equivalent to the desired inequality.

\end{proof}

Let $\odd(m)$ be the largest odd divisor of $m$. From Lemma \ref{precursor to no tricky specials} above we have the following. 

\begin{proposition} There is no prime $q$ such that $H(\odd(q+1)) < 2$, and $H(q\odd(q+1)) > 2$. \label{no tricky specials}
 \end{proposition}

Proposition \ref{no tricky specials} says that if one can make an abundancy argument to rule out prime $q$ using $q$'s contribution, then one did not in fact need to use $q$ at all. That is, one can rule out $q$ being special based solely on  using abundancy considerations to show that there there is no odd perfect number with $\sigma(\frac{p+1}{2}) \mid n$.

For higher powers the situation is slightly different. For example, $H(\sigma(5^5)) <2$, but $H(5\sigma(5^5) > 2$. In fact, one can use an abundancy argument to show that one cannot have any odd perfect number $n$ with  $5^5 \mid\mid n$, even though a simple abundancy argument fails to be sufficient to show that $\frac{\sigma(5^5)}{2} \nmid n$.  A variant of this argument seems to have been first explicitly noted by Cohen and Sorli \cite{Sorli and Cohen}, although it was implicitly used in prior work.  

When the exponent is even, similar examples are more difficult to come by but they do exist. Consider for example $7^{944}$. One has that $H(\sigma(7^{944}))<2 $ but $H(7\sigma(7^{944}))>2$. A little work shows that one cannot have an odd perfect number $n$ where $7^{944} \mid \mid n$.  The exponent $944$ arises because $\sigma(7^{944}) = \frac{7^{945} -1}{6}$, and since $945=(3^3)(5)(7)$ we get many small prime factors in  $\sigma(7^{944})$.

The discussion in the previous three paragraphs motivates three questions.

\begin{question} Does there exist an odd prime $p$ such that $H(\sigma(p^2)) < 2$, but where $H(p\sigma(p^2)) >2$? \label{Square question}
\end{question}

If Question \ref{Square question} has an affirmative answer, then Lemma \ref{precursor to no tricky specials} is essentially best possible.  Note that we do not need to take the odd part in the above question unlike in Lemma \ref{precursor to no tricky specials}  since $p^2+p+1$ is always odd.)

\begin{question} Let $\Phi_m(x)$ be the $m$th cyclotomic polynomial. Is there any $m$ and odd prime $p$ such that 
$H(\odd(\Phi(p)) < 2$ and $H(p\odd(\Phi(p)) > 2$? 
\label{Cyclotomic question}
\end{question}

An affirmative answer to Question \ref{Square question} would imply a positive answer to Question \ref{Cyclotomic question}, since it would be the case when $m=2$.

\begin{question} What is the smallest $n$ such that $n=p^a$ where $p$ is an odd prime, $a$ is even, and where $H(n)< 2$ and $H(p\sigma(n)) >2$?
\end{question}

There is a thematic connection between this question and prior work where Cohen and Sorli \cite{Sorli and Cohen} used essentially a similar abundancy argument to show that if $p^a$ is a component in an odd perfect number and $p$ is the special
prime, then $\sigma(\sigma(p^a)) \leq 3p^a - 1$. They also showed that if $p$ is not special, then  $\sigma(\sigma(p^a)) \leq 2p^a - 2$.

Dris \cite{Dris}    showed that if $n=m^2p_j^{a_j}$ and $p_j > m$, then $a_j=1$. We use this opportunity to show an alternate proof of this result using connected ideas to the above which also strengthens it slightly.

\begin{proposition} Assume that $n$ is an odd perfect number where
$n=m^2q^e$, where $(q,m)=1$ and $q \equiv e \equiv 1 \pmod{4}$ . Then either $e=1$, or $$q < 2^{\frac{1}{5}}m^{\frac{2}{5}}.$$  \label{Tightened Dris prop}
\end{proposition}
\begin{proof} Assume that $e > 1$. Thus, $e \geq 5$. Since $\gcd(q,\sigma(q^e))=1$, we must have $q^5 \mid \sigma(m^2)$, and so 
$$q^5 \leq \sigma(m^2) < 2m^2,$$
which implies the desired inequality. \qedhere \end{proof}
We can improve the constant in Proposition \ref{Tightened Dris prop} slightly more with only a small amount of effort.

\begin{proposition}  Assume that $n$ is an odd perfect number where
$n=m^2q^e$ where $(q,m)=1$ and $q \equiv e \equiv 1 \pmod 4$. Then either $e=1$, or $$q < \left(\frac{2}{(3^2)(5^2)(11^2)(13^2)} \right)^{\frac{1}{5}}m^{\frac{2}{5}}.$$
\end{proposition}
\begin{proof} Assume that we have $e > 1$. We now consider two cases, $e=5$ and $e \geq 9$.  Assume that $q=5$. It follows from  results of Nielsen \cite{Nielsen} that $m$ must have at least $9$ distinct prime factors. This means that there have to be at least four primes $p_1$, $p_2$, $p_3$, $p_4$ with exponents $a_1$, $a_2$, $a_3$, $a_4$ such that for $i=1,2,3,4$ we have  $p_i^{a_i} \mid\mid m^2$, and $q \nmid \sigma(p_i^{a_i})$.

Thus we have 
\begin{equation}
    q^5\mid  \sigma\left(\frac{m^2}{p_1^{a_1}p_2^{a_2}p_3^{a_3}p_4^{a_4}}\right). \label{Tightened q5 inequality with four primes} 
\end{equation}

Inequality \eqref{Tightened q5 inequality with four primes} then implies that 

\begin{equation}
    q^5 < 2\frac{m^2}{p_1^{a_1}p_2^{a_2}p_3^{a_3}p_4^{a_4}}. \label{q5 tightened upper bound} \end{equation}
Since no odd perfect number can be divisible by 3, 5, and 7 we must have $p_1^{a_1}p_2^{a_2}p_3^{a_3}p_4^{a_4} \geq (3^2)(5^2)(11^2)(13^2)$, which together with Inequality \eqref{q5 tightened upper bound} implies the inequality in question.
If $e \geq 9$, then we have by the same logic as the proof of the previous proposition that $q^9  < 2m^2$, which is weaker than the desired inequality unless $m < 48289774$, and there is no odd perfect number with such a small $m$ value. 
    
\end{proof}

The constant in the previous proposition can be improved further by a variety of means, such as by using that no odd perfect number is divisible by $3,5,11$ and $13$. Similarly, one can use that if $p$ is a prime where $p^a\mid\mid m^2$, and $(q,\sigma(p^a))=1$, then there must be another prime power $r^b$ with $r^b||m$ and $q|\sigma(r^b)$ to further control $h(m^2)$. However, this involves arduous case checking followed by taking a $5$th root. Thus, improvement in the constant ends up being small proportional to the large amount of effort of such a calculation.

Note that in this context, the weakest bounds in our main inequalities occur when $p_j$ is essentially as small as possible. In contrast, in many other contexts, a small $p_j$ results in tighter bounds. For example, in Acquaah and Konyagin's proof \cite{AK} of an upper bound on the largest prime factor, the most difficult situation is when $p_j$ is the largest prime factor. Similar remarks apply to the bounds on the second and third largest prime factor. There are other bounds in the literature which bound  $p_1p_2 \cdots p_k$ above such as those of Klurman \cite{Klurman} and similar bounds by Luca and Pomerance \cite{Luca and Pomerance}. In both of those papers, the situation where $p_j$ is the largest prime factor  gives the weakest bounds. 

Two final remarks: First, for  any positive integer $n$ with smallest prime factor $p_1$, \begin{equation}D(n) \leq \frac{n \log_{p_1} n}{p_1}.\label{DOL bound}\end{equation}
See  Dahl, Olsson, and Loiko \cite{DOL} as well Pasten \cite{Pasten1} for this and related bounds. One obvious question is how much Inequality \eqref{DOL bound} can be tightened under the assumptions that $n$ is an odd perfect number or the weaker assumption that $n$ is a primitive non-deficient number. The author was unable to obtain non-trivial improvements of this inequality under these assumptions, but it seems plausible that such improvements are possible. \\

Second, since Proposition \ref{upper bound for T, between all prim defs and opns} and Proposition \ref{first lower upper bound for T} give both upper and lower bounds for $T(n)$ close to $\log 2$, while $T(n)$ is itself the sum of distinct primes, this suggests a variant of Diophantine approximation, where one is interested in approximating a given real number $\alpha$ by a finite sum of reciprocals of distinct primes. Results bounding how tight such approximations for $\log 2$ could potentially offer insight into tightening Proposition \ref{upper bound for T, between all prim defs and opns} and Proposition \ref{first lower upper bound for T}.\\

{\bf Acknowledgements} Pace Nielsen and Eve Zelinsky both made helpful corrections to the manuscript.

2020 AMS Subject Classification: Primary 11A25; Secondary 11N64
Keywords: Perfect number, non-deficient number, abundant number, sum-of-divisors

\end{document}